\documentclass[12pt,a4paper]{amsart}
\usepackage{amsfonts}
\usepackage{amsthm}
\usepackage{amsmath}
\usepackage{amscd}
\usepackage[latin2]{inputenc}
\usepackage{t1enc}
\usepackage[mathscr]{eucal}
\usepackage{indentfirst}
\usepackage{graphicx}
\usepackage{graphics}
\usepackage{pict2e}
\usepackage{epic}
\usepackage{amssymb}
\usepackage[colorlinks,citecolor=red,urlcolor=blue,bookmarks=false,hypertexnames=true]{hyperref} 
\numberwithin{equation}{section}
\usepackage[margin=3cm]{geometry}

\theoremstyle{plain}
\newtheorem{Th}{Theorem}[section]
\newtheorem{Lemma}[Th]{Lemma}
\newtheorem{Cor}[Th]{Corollary}
\newtheorem{Prop}[Th]{Proposition}
\newtheorem*{Theorem-non}{Theorem}
\newtheorem*{Theorem-non2}{Theorem}

 \theoremstyle{definition}
 \newtheorem*{Proof-non}{Proof of Theorem \ref{Maintheorem} assuming Propositions \ref{Prop1},\ref{Propm}}
\newtheorem*{Proof-non2}{Proof of (1)  ($\bf{m_{1}}$-estimate) in Proposition \ref{Propm} assuming Proposition \ref{Proposition 5.1}}
\newtheorem*{Proof-non3}{Proof of Theorem \ref{Maintheorem2} assuming Propositions \ref{Prop1},\ref{Propm}}
\newtheorem*{Proof-non4}{Proof of Proposition \ref{Prop1}}
\newtheorem*{Proof-non5}{Proof of Proposition \ref{Propm}}

\newtheorem{Def}[Th]{Definition}

\newtheorem{Rem}[Th]{Remark}
\newtheorem{?}[Th]{Problem}

\newcommand{\RomanNumeralCaps}[1]
    {\MakeUppercase{\romannumeral #1}}
\begin{document}

\author{Jiseong Kim}
\address{The University of Mississippi, Department of Mathematics
Hume Hall 335
Oxford, MS 38677}
\email{Jiseongk51@gmail.com}
\title{on Asymptotics of Shifted Sums of Dirichlet convolutions}

\begin{abstract} 
The objective of this paper is to obtain asymptotic results for shifted sums of multiplicative functions of the form $g \ast 1$, where the function $g$ satisfies the Ramanujan conjecture and has conjectured upper bounds on square moments of its L-function. We establish that for $H$ within the range $X^{23/24+10\varepsilon} \leq  H \leq X^{1-\varepsilon}$, there exist constants $B_{f,h}$ such that $$ \sum_{X\leq n \leq 2X} f(n)f(n+h)-B_{f,h}X=O_{f,\varepsilon}\big(X^{1-\varepsilon^{2}/4}\big)$$ for all but $O_{f,\varepsilon}\big(HX^{-\varepsilon^{2}/3}\big)$ integers $h \in [1,H].$
Our method is based on the Hardy-Littlewood circle method. In order to treat minor arcs, we use the convolution structure and a cancellation of $g(n)$ that are additively twisted, applying some arguments from a paper of Matom\"a{}ki, Radziwi\l{}\l{} and Tao. 
Also, we establish an upper bound for weighted exponential sums, which may be of independent interest.

\end{abstract}

\maketitle

\section{\rm{Introduction}}
When we deal with arithmetic functions in analytic number theory, an interesting task is to estimate the correlation expressed as $$\sum_{n \leq X} a(n)a(n+h)$$ where $a:\mathbb{N} \rightarrow \mathbb{C}$ is an arithmetic function. Also, this problem has attracted substantial interest when $a(n)$ is a coefficient of an L-function, due to its relevance in estimating moments of L-functions and establishing subconvexity bounds (for example, see \cite{Blomerv}). 

Let $\mathbb{U}$ denote the unit disc. When $a(n)$ is a bounded multiplicative function $a:\mathbb{N} \rightarrow \mathbb{U},$ we have a relatively good understanding of how to estimate these correlations. For example, in \cite{Klurman}, it is proved that for any pretentious bounded multiplicative function $a: \mathbb{N} \rightarrow \mathbb{U}$ such that a certain modified distance between $a(n)$ and $1$ is small, we have 

$$\sum_{n \leq x} a\left(P(n)\right) a\left(Q(n)\right) = x \prod_{p \leq x} \left(\frac{1}{x}\lim_{x \rightarrow \infty} \sum_{n \leq x} a_{p}\left(P(n)\right) a_{p}\left(Q(n)\right)\right)+O(1/\log \log X)$$ where $P,Q \in \mathbb{Z} [x],$ and $a_{p}\left(q^{k}\right)=a(q^{k})$ if $q=p,$ and $1$ otherwise (for details, see \cite[Theorem 1.3]{Klurman}). Note that the results proven in \cite{Klurman} address more general cases. However, when $a(n)$ is not bounded, tackling the problem becomes notably more challenging. a classical example of unbounded cases is the divisor function. Let $$ d_{k}(n):= \sum_{a_{1}a_{2}...a_{k}=n, \atop a_{i} \in \mathbb{N}} 1.$$  Deshouillers and Iwaniec \cite{DeshIwan} proved that 
$$
\sum_{X \leq n \leq 2X} d_{2}(n)d_{2}(n+h)=P_{2}(\log X)X +O(X^{\frac{2}{3}+o(1)})
$$
as $X \rightarrow \infty$, where $P_{2}(x)$ is a quadratic polynomial with coefficients depending on $h$. For $k=3,$ Baier, Browning, Marasingha, and Zhao \cite{BBMZ} proved that when $X^{\frac{1}{3}+\epsilon} \leq H \leq X^{1-\epsilon},$ for all but $O_{\epsilon}(HX^{-\delta})$ integers $|h|<H,$ 
\[
\sum_{X \leq n \leq 2X} d_{3}(n)d_{3}(n+h)=P_{3}(\log X)X + O(X^{1-\delta})
\]
where $P_{3}(x)$ is a quartic polynomial with coefficients depending on $h$, and for some $\delta>0.$ For the higher-order divisor functions, Matom\"a{}ki, Radziwi\l{}\l{}, and Tao proved the following averaged divisor correlation conjecture \cite{MRT1}. 

\begin{Theorem-non} \cite[Theorem 1.3, (ii)]{MRT1}
Let $A>0$, and let $0<\epsilon< \frac{1}{2}$. Let $k,l \geq 2$ be fixed. Suppose that $X^{\frac{8}{33}+\epsilon} \leq H \leq X^{1-\epsilon}$ for some $X\geq 2$. Let $0\leq h_{0} \leq X ^{1-\epsilon}$.
Then for each $h,$ there exists a polynomial $P_{k,l,h}$ of degree $k+l-2$ such that
$$
\sum_{X \leq n \leq 2X} d_{k}(n)d_{l}(n+h) = P_{k,l,h}(\log X)X + O_{A,\epsilon, k,l}\big(X(\log X)^{-A}\big) \quad \textrm{as} \quad X \rightarrow \infty
$$
for all but $O_{A,\epsilon,k,l}\big(H (\log X)^{-A}\big)$ values of $h$ with $|h-h_{0}| \leq H$. 
\end{Theorem-non} 
 Their methods apply to situations where arithmetic functions can be expressed as Dirichlet convolutions involving other multiplicative functions that have good cancellation (see \cite{MRT1}). This method also shares a similar philosophy, assuming the Siegel-Walfisz property on one of the factors (for example, see \cite{FouvryTenenbaum}).
Building upon their techniques, in \cite{Kim2022s}, we obtained the asymptotic behavior of correlations involving the squares of Hecke eigenvalues of holomorphic cusp forms. In another direction, one may be interested in \cite{Matthiesen}. Recently, in \cite{MSTJ2023} and its sequel, the authors proved various correlations of arithmetic functions that involve certain types of combinatorial decompositions, such as Heath-Brown-type identities, where the factors are either $1$, $\mu$, or the von Mangoldt function $\Lambda$. Their methods are effective when one can control certain non-periodic factors (such as $\mu$ or $\Lambda$).

In this paper, we explore an alternative approach that considers a different class of multiplicative functions, which may not be accessible through their methods.
\begin{Def} For $k \in \mathbb{N}$, let $\mathcal{F}_{k}$ denote the class of multiplicative functions $g:\mathbb{N} \rightarrow \mathbb{C}$ that satisfy the following properties:
\begin{enumerate}
    \item (Ramanujan Conjecture) Let $L(g,s):=\sum_{n=1}^{\infty} \frac{g(n)}{n^{s}}.$ $L(g,s)$ has an Euler product 
    $$ \prod_{i=1}^{k} \prod_{p} (1-\alpha_{i}/p^{s})^{-1}$$ for some $|\alpha_{i}| = 1.$ From this, it is easy to see that $g(n) \leq d_{k}(n).$ 
    \item  Let $\displaystyle L(g,\chi,s):= \sum_{n=1}^{\infty} \frac{g(n)\chi(n)}{n^{s}}$, where $\chi$ is a Dirichlet character modulo $q.$ The function $L(g,\chi,s)$ is holomorphic and has an analytic continuation for $\Re(s)>1/2$. It is nonzero when $\Re(s)\geq 1$ and does not have any poles.
Furthermore, it satisfies the conjectured bound such that 
\begin{equation}\label{squareup}\int_{T}^{2T} \left|L(g,\chi,\sigma+it)\right|^{2}dt \ll_{\varepsilon} q^{\varepsilon^{2}}\left(T+1\right)^{1+\varepsilon^{2}},\end{equation} for any $\varepsilon>0$, $1/2 \leq \sigma <1$, and $T>0$. 
\end{enumerate}
\end{Def} 
For details on the conjectured upper bounds on the square moments of L-functions, see \cite{Conrey}. Note that the conjectured upper bound on the square moments is a much weaker condition than the generalized Lindel\"of  hypothesis or the generalized Riemann hypothesis. Assuming 
\begin{equation}\int_{T}^{2T} \left|L(g,\chi,\sigma+it)\right|^{2}dt \ll_{\varepsilon} q^{1-\delta} \left(T+1\right)^{1+\varepsilon^{2}}\nonumber\end{equation} for some $\delta>0,$ instead of $\eqref{squareup}$, still allows us to obtain weaker versions of Theorem \ref{Maintheorem}.
\begin{Rem}
Note that one can remove the condition on the pole since the condition is only relevant to major arc estimates, and it could be resolved easily.
\end{Rem}
\begin{Def} let $\mathcal{E}_{k}$ denote the class of multiplicative functions $f:\mathbb{N} \rightarrow \mathbb{C}$ such that
$$f(n) =  g \ast 1 (n) := \sum_{d|n} g(d)$$ for some $g \in \mathcal{F}_{k}.$ Note that $f(n) \leq d_{k+1}(n).$
\end{Def} \noindent In this paper, we prove the following theorem.
\begin{Th}\label{Maintheorem} Let $\varepsilon$ be a fixed small positive constant. Let $f \in \mathcal{E}_{k}.$ Suppose $X^{23/24+10\varepsilon} \leq$ $H \leq X^{1-\varepsilon}$. There are constants $B_h$ such that as $X \rightarrow \infty$,
$$\sum_{X \leq n \leq 2 X} f(n) f(n+h)-B_h X=O_{f,k, \varepsilon}\left(X^{1-\varepsilon^{2}/4}\right)$$
for all but $O_{f,k,\varepsilon}\left(HX^{-\varepsilon^{2}/3}\right)$ integers $h \in[1, H]$.
\end{Th}
\begin{Rem} 
Here, we discuss some methods from previous works \cite{MRT1}, \cite{MRT2} and explain why our $H$ is significantly larger than in those works. For divisor functions or Mobius functions, these functions can be regarded as linear combinations of Dirichlet convolutions, whose factors all satisfy good cancellation properties. However, in general, as the squares of Hecke eigenvalues, only one factor satisfies a good cancellation property.

These cancellation properties can be applied in various ways, and even for the factors of divisor functions (which is just the indicator function), the upper bounds of the higher moments of Dirichlet \( L \)-functions are relatively well-studied. Since such results have not yet been proved for general functions, we need to treat certain parts of Dirichlet convolutions differently.

In our previous work \cite{kim2021hecke}, we used Miller's bound on weighted exponential sum estimates, which saves \( X^{-1/4} \), allowing us to obtain asymptotics for the shifted sum of Hecke eigenvalue squares with $ H > X^{2/3} $. Here, we consider more general cases, so our exponential sum estimates do not yield such a strong upper bound.

The method in \cite{MRT2} has very small savings in error terms and highly depends on Shiu's theorem, which is not suitable for many cases. For example, the square of a Hecke eigenvalue has a constant mean value, but the upper bound of individual Hecke eigenvalues is given by the divisor function. Since this individual bound is much larger than the mean value, the method in \cite{MRT2} for minor arc estimates, especially for large value estimates, does not work well.

\end{Rem}

Let $\Omega$ be a GL(n) Maass cusp form, and let $A_{\Omega}(d,1,1, \dots, 1)$ denote the normalized Fourier coefficients of $\Omega$. Define $B_{\Omega}(m) := \sum_{d \mid m} A_{\Omega}(d, 1, 1, \dots, 1)$, which is the $m$-th coefficient of the $L$-function associated with $1 \boxplus \Omega.$ Assuming that $m \rightarrow A_{\Omega}(m, 1, 1, \dots, 1) \in \mathcal{F}_{k}$, we obtain the following corollary from Theorem \ref{Maintheorem}.

\begin{Cor}
Assume that $m \rightarrow A_{\Omega}(m,1,1,...,1) \in \mathcal{F}_{k}.$ Let $\varepsilon$ be a fixed small positive constant. Suppose $X^{23/24+10\varepsilon} \leq$ $H \leq X^{1-\varepsilon}$. There are constants $\mathcal{B}_h$ such that as $X \rightarrow \infty$,

$$\sum_{X \leq n \leq 2 X} B_{\Omega}(n) B_{\Omega}(n+h)-\mathcal{B}_h X=O_{\Omega,k, \varepsilon}\left(X^{1-\varepsilon^{2}/4}\right)
$$
for all but $O_{\Omega,k,\varepsilon}\left(HX^{-\varepsilon^{2}/3}\right)$ integers $h \in[1, H]$.
\end{Cor}

\subsection{Sketch of the proof}
In \cite{MRT1}, the authors applied the good-cancellation property while dealing with estimates for minor arcs through a combinatorial decomposition. This property allows us to achieve cancellations of powers of $\log X$, and the constant function $1$ satisfies this property.
To provide some context for our proof, let us briefly revisit our previous work on the squares of Hecke eigenvalues. These squares, denoted as $\lambda(n)^{2}$, exhibit a convolution structure known as the Hecke relations:
$$\lambda(n)^{2} = \sum_{d|n} \lambda(d^{2}) = \lambda(d^{2}) \ast 1.$$
These relations involve Hecke eigenvalues and the constant function $1$.  However, because we are uncertain whether Hecke eigenvalues have good cancellation, we 
 rely on obtaining sufficient cancellation from the constant function $1$. Consequently, this limitation restricts us to consider Hecke eigenvalues over relatively short intervals.  To address this, we applied Miller's work \cite{Miller1} on the cancellation of the coefficients of symmetric square L-functions that are additively twisted:
 $$\sum_{n=1}^{X} \lambda(n^{2})e(n\alpha)=O_{f,\varepsilon}\left(X^{3/4+\varepsilon} \right).$$
 The methods we employed can be extended to other general multiplicative functions $a(n)$ expressed as $a(n) = b \ast 1(n)$, and similar cancellations in
 $$ \sum_{n=1}^{X} b(n) e(n\alpha).$$
 However, it's good to note that obtaining such cancellations uniformly can be very challenging in general cases. A standard approach to study 
$$ \sum_{n=1}^{X} b(n)e(n \alpha)$$ is to use the Mellin transform, applying the analytic properties of L-functions.  For this, we apply the conjectured upper bounds on square moments of L-functions. This method provides a good upper bound when the parameter $\alpha$ isn't excessively large. 
Since $$\sum_{n=1}^{X}a(n)e(n\alpha)= \sum_{n=1}^{X}  \sum_{m=1}^{X/n} b(m)e(mn \alpha),$$ we can't directly apply the upper bound from the Mellin transform because the inner sum depends on $n$, and when $n$ becomes sufficiently large, the parameter $n \alpha$ becomes too large. However, as mentioned before, we can handle cases when $n$ is sufficiently large, by relying on sufficient cancellation from the constant function $1$. Consequently, we only need to focus on cases where $n$ isn't extremely large, allowing us to get sufficient cancellation from the upper bound obtained through the Mellin transform.

\begin{Rem}\label{Remark 3.3}
Let $\displaystyle q=q_{0}q_{1}, q_{0}=\prod_{i=1}^{l} p_{i}^{\beta_{i}}$ for some primes $p_{i}$. Let $$D(s):= \sum_{n=1 \atop (n,q_{1})=1}^{\infty} \frac{f(q_{0}n)}{n^{s}}.$$
 Then
$$D(s)=L(g,\chi,s) \zeta(s) \times R_{q_{0},q_{1}}(s)$$
 where
$$R_{q_{0},q_{1}}(s):= \prod_{p|q} (1-1/p^{s}) \prod_{i=1}^{k} (1-\alpha_{i}/p^{s}) \prod_{i=1}^{l} (f(p_{i}^{\beta_{i}})+f(p_{i}^{\beta_{i}+1})p_{i}^{-s}+...),$$ and 
$\chi$ is the principal character modulo $q_{1}. $
By using the Ramanujan conjecture, it is easy to see that
$$ \prod_{p|q} (1-1/p^{1/2+it}) \prod_{i=1}^{k} (1-\alpha_{i}/p^{1/2+it}) \ll \prod_{p|q} (1+1/p^{1/2})^{k-1} \ll  (\ln q)^{k-1}. $$
Using the divisor bound $$d_{k+1}\left(p^{\beta}\right)= \binom{\beta+k}{k}\ll_{k} \beta^{k} \prod_{j=1}^{k}\left(1+\frac{k}{\beta}\right)\ll_{k} \beta^{k}, $$ we have 
\begin{equation}\begin{split}\prod_{i=1}^{l} (f(p_{i}^{\beta_{i}})+f(p_{i}^{\beta_{i}+1})p_{i}^{-1/2+it}+...)&\ll_{k} \prod_{i=1}^{l} \beta_{i}^{k}\left(p_{i}^{1/2}/(p_{i}^{1/2}-1)\right) 
\\&\ll_{k} q_{0}^{1/2} \log_{2}(q_{0})^{k} \prod_{i=1}^{l} (1/ (p_{i}^{1/2}-1))
\\&\ll_{k}  q_{0}^{1/2} \log_{2}(q_{0})^{k+1} .\end{split}\end{equation}
Therefore, we have
\begin{equation}\label{3Rq0q1} \begin{split} R_{q_{0},q_{1}}(\frac{1}{2}+it) &\ll_{k} q_{0}^{1/2+o(1)} (\ln q)^{k-1}. \end{split}\end{equation}
Since we assume the upper bound of the square moments of $L(g,\chi,s),$ we have

\begin{equation}\label{q1}\begin{split}\int_{T}^{2T} \left|\frac{D(\sigma+it)}{\zeta(1/2+it)}\right|^{2}dt \ll_{f,k,\varepsilon} 
T^{1+\varepsilon^{2}}q_{0}^{1+o(1)} (\ln q)^{2k-1}.  \end{split}\end{equation}
By using a similar argument, when $\chi$ is a Dirichlet character modulo $q_{1},$ we have
\begin{equation}\label{q2}\int_{T}^{2T} \left|\frac{D(\sigma+it,\chi)}{L(\sigma+it,\chi)}\right|^{2}dt \ll_{f,k,\varepsilon} 
T^{1+\varepsilon^{2}}q_{0}^{1+o(1)} (\ln q)^{2k-2} \end{equation} where
$$D(s,\chi)=\sum_{n=1}^{\infty} \frac{\chi(n)f(q_{0}n)}{n^{s}}.$$
\end{Rem}

\begin{Lemma}\label{Lemma13}
Let $g \in \mathcal{F}_{k}$, and let $0 \leq a < q$. Then, for any $ \alpha \geq \frac{1}{X^{5/6+2\varepsilon}}$, we have
$$ \sum_{m=1}^{X} g(m) e\left(m\left(\frac{a}{q}+\alpha\right)\right) \ll_{g,k,\varepsilon} q^{1/2+\varepsilon^{2}}X^{1+\varepsilon^{2}} |\alpha|^{\frac{1}{2}+\varepsilon^{2}} + X^{5/6+6\varepsilon}$$
\end{Lemma}

\begin{proof}
Let $A(s) := \sum_{n=1}^{\infty} g(n) e\left(\frac{an}{q}\right)n^{-s}$. For simplicity, assume that $(a, q) = 1$. 

It is known that when $(an, q) = 1$, we have
$$ e\left(\frac{an}{q}\right) = \frac{1}{\phi(q)} \sum_{\chi \pmod{q}} \tau(\bar{\chi}) \chi(an), $$
where 
$$ \tau(\chi) := \sum_{1 \leq m \leq q \atop (m, q) = 1} \chi(m) e\left(\frac{m}{q}\right) $$
is the Gauss sum. Therefore, we can write
\begin{equation}\label{additive}
\sum_{n=1}^{\infty} g(n) e\left(\frac{an}{q}\right) n^{-s} = \sum_{q = q_0 q_1} \frac{1}{\phi\left(q_1\right) q_0^s} \sum_{\chi \pmod{q_1}} \tau(\bar{\chi}) \chi(a) \sum_{\substack{n=1 \\ (n, q_1) = 1}}^{\infty} g(q_0 n) \chi(n) n^{-s}.
\end{equation}

Now, let $\psi(x)$ be a smooth, compactly supported function such that the support is the interval $[1 - X^{-1 + \beta}, 2 + X^{-1 + \beta}]$, satisfying $\psi(x) = 1$ for $1 \leq x \leq 2$, $\psi(x) \leq 1$ otherwise, and 
$$ \psi^{(j)}(x) \ll_{j} X^{(1 - \beta)j}, $$ 
where $\beta$ will be chosen later. Define 
$$ B(s, \alpha) = \int_0^{\infty} \psi\left(\frac{x}{X}\right) e(-x\alpha) x^{s - 1} \, dx. $$

Using the Mellin transform, we see that
$$ \sum_{m=1}^{\infty} g(m)e\left(m\left(\frac{a}{q} + \alpha\right)\right)\psi\left(\frac{m}{X}\right) = \frac{1}{2\pi i} \int_{(\sigma')} A(s) B(s, \alpha) \, ds$$
where $\sigma ' = 1 + \frac{1}{\log X}.$
By applying \eqref{additive}, we have
\begin{equation}\label{absa} \int_{(\sigma')} A(s) B(s, \alpha) \, ds = \sum_{q = q_0 q_1} \sum_{\chi \pmod{q_1}} \tau(\bar{\chi}) \chi(a) \frac{1}{\phi\left(q_1\right)} \int_{(\sigma')} \sum_{\substack{n=1 \\ (n, q_1) = 1}}^{\infty} g(q_0 n) \chi(n) (q_0 n)^{-s} B(s, \alpha) \, ds. \end{equation}
Now, let
$$ L(g, \chi, s; q_0, q_1) = \sum_{\substack{n=1 \\ (n, q_1) = 1}}^{\infty} g(q_0 n) \chi(n) n^{-s}. $$

Thus, the right-hand side of \eqref{absa} is
$$ \sum_{q = q_0 q_1}  \sum_{\chi \pmod{q_{1}} \atop \chi \neq \chi_{0}} \tau(\bar{\chi}) \chi(a) \frac{1}{\phi\left(q_1\right)q_0^s} \int_{(\sigma)} L(g, \chi, s; q_0, q_1) B(s, \alpha) \, ds + O\left( \frac{q^{\varepsilon}}{q\alpha} \right) $$
where $\frac{1}{2} \leq \sigma <1,$ the error term arises from the simple pole of $L(g,\chi_{0},s;q_{0},q_{1}).$ Note that when $B(s,\alpha) \ll_{A} t^{-A}$ for sufficiently large $t$ depends on $X$ (see \eqref{smoothed}), so when shifting the line of integration, the contribution from the horizontal segment is negligible.

Next, we split the integral into two parts:
$$ \int_{|t| \geq 3 \pi \alpha X} + \int_{|t| \leq 3 \pi \alpha X} L(g, \chi, s; q_0, q_1) B(s, \alpha) \, d\alpha. $$

For the first integral, we use the result from \cite[Lemma 8.1]{1BKY}, noting that $\left(\psi\left(\frac{x}{X}\right)\right)^{(j)} \ll X^{-j\beta}$. Let $h(x) = t \log x - 2\pi \alpha x$. To show that the first integral is negligible, we use the upper bound in \cite[Lemma 8.1]{1BKY} which is \begin{equation}\label{smoothed} B(s,\alpha) \ll_{A} X^{\sigma'}\min\left(\frac{X  \frac{t}{X}}{\sqrt{t}},\frac{t}{X} X^{\beta}\right)^{-A}, \;\; \; \textrm{for   any  A>0},\end{equation} choosing $\beta = 5/6+5\varepsilon$ to satisfy the conditions given in \cite[Lemma 8.1]{1BKY}. Then by \cite[Lemma 8.1]{1BKY}, the first integral is negligible
(following their notations, $Y=t, X=1, Q= X, U=X^{\beta}, R= \frac{t}{X}.)$

Splitting the square of $B(s,\alpha)$, we have 
\begin{equation}\begin{split}
\int_{|t| \ll  \alpha X} & \left|B(\sigma + it, \alpha)\right|^{2} dt \\& = \int_{|t| \ll  \alpha X} \int_{0}^{\infty} \int_{0}^{\infty} \psi\left(\frac{x}{X}\right) \psi\left(\frac{y}{X}\right) e(-\alpha(x - y))\left(\frac{x}{y}\right)^{it} (xy)^{\sigma - 1} dx \, dy \, dt.
\end{split}\end{equation}

By interchanging the integrals, we have
\begin{equation}
\int_{|t| \ll  \alpha X}  \left|B(\sigma + it, \alpha)\right|^{2} dt \ll X^{2\sigma - 2} \int_{X - X^{\beta}}^{2X + X^{\beta}} \int_{X - X^{\beta}}^{2X + X^{\beta}} \min\left(\alpha X, \frac{1}{|\log(x/y)|}\right) dy \, dx.
\end{equation}

When $\left|x - y\right| < 1/\alpha$, we take $\alpha X$ as the bound for the minimum. Hence,
\begin{equation}
\int_{X - X^{\beta}}^{2X + X^{\beta}} \left(\int_{x - \frac{1}{\alpha}}^{x + \frac{1}{\alpha}} \alpha X \, dy \right) dx \ll X^2.
\end{equation}

Additionally,
\begin{equation}
\begin{split}
\int_{X - X^{\beta}}^{2X + X^{\beta}} \int_{X - X^{\beta}}^{x - \frac{1}{\alpha}} + \int_{x + \frac{1}{\alpha}}^{2X + X^{\beta}} \frac{1}{|\log(x/y)|} \, dy \, dx 
&\ll \int_{X - X^{\beta}}^{2X + X^{\beta}}\left| \left( \int_{X - X^{\beta}}^{x - \frac{1}{\alpha}} + \int_{x + \frac{1}{\alpha}}^{2X + X^{\beta}} \frac{y}{|x - y|} \, dy \right) \right|dx \\
&\ll \int_{X - X^{\beta}}^{2X + X^{\beta}} X \log X \, dx \\
&\ll X^2 \log X.
\end{split}
\end{equation}

Therefore, we conclude that
\begin{equation}
\int_{|t| \ll  \alpha X}  \left|B(\sigma + it, \alpha)\right|^{2} dt \ll X^{2\sigma} \log X.
\end{equation}

For the last integral, by applying H\"{o}lder's inequality and \eqref{q2}, we have
\begin{equation}
\begin{split}
\int_{|t| \ll \alpha X} & L(g, \chi, \sigma + it; q_0, q_1) B(\sigma + it, \alpha) \, ds \\
&\ll \left( \int_{|t| \ll \alpha X} \left|L(g, \chi, \sigma + it; q_0, q_1)\right|^2 \, dt \right)^{1/2} \left( \int_{|t| \ll \alpha X} \left|B(\sigma + it, \alpha)\right|^2 \, dt \right)^{1/2} \\
&\ll (\alpha X)^{1/2 + \varepsilon^2 / 2} X^\sigma (\log^{1/2} X) q_0^{(1 + \varepsilon^2) / 2} q^{o(1)} \\
&\ll \alpha^{1/2 + \varepsilon^2 / 2} X^{\sigma + 1/2 + \varepsilon^2} q_0^{(1 + \varepsilon^2) / 2} q^{o(1)}.
\end{split}
\end{equation}

Thus, we obtain
\begin{equation}
\begin{split}
\int_{(\sigma)} A(s) B(s) \, ds &\ll_{g, k, \varepsilon} \sum_{q = q_0 q_1} \sum_{\chi \pmod{q_1}  \atop \chi \neq \chi_{0}} |\tau(\bar{\chi})| \frac{1}{\phi(q_1)} \frac{q^{o(1)}}{q_0^{\sigma - 1/2 - \varepsilon^2 / 2}} \left( X^{\sigma + 1/2 + \varepsilon^2 / 2} \alpha^{1/2 + \varepsilon^2 / 2} \right).
\end{split}
\end{equation}

Taking $\sigma = 1/2$ and using the bound $\tau(\bar{\chi}) \ll q_1^{1/2}$, this is  bounded by
$$ q^{1/2 + \varepsilon^2} X^{1 + \varepsilon^2 / 2} \alpha^{1/2 + \varepsilon^2 / 2}. $$

After removing the weight function $\psi$ and applying Shiu's theorem \cite{Shiu1}, we have
\begin{equation}
\begin{split}
\sum_{m=1}^{\infty} & g(m) e\left(m\left(\frac{a}{q} + \alpha\right)\right) \psi\left(\frac{m}{X}\right) - \sum_{m=X}^{2X} g(m) e\left(m\left(\frac{a}{q} + \alpha\right)\right) \\
&\ll_\varepsilon \sum_{X - X^{\beta} \leq m \leq X} d_k(m) + \sum_{2X \leq m \leq 2X + X^\beta} d_k(m) \\
&\ll_\varepsilon X^{\beta} \log^{k-1} X.
\end{split}
\end{equation}

\end{proof}
\begin{Rem} Note that, with a slight modification to the argument in the proof of Lemma \ref{Lemma13}, one can obtain a short interval version of the result. We will apply it in an upcoming paper to study triple correlations.
\end{Rem}

\subsection*{Acknowledgements} The author is grateful to the anonymous referee for their helpful comments, especially for pointing out the error in Lemma 1.8 in the previous version. The author would also like to thank Kunjakanan Nath for valuable discussions.

\section{Notations}
In this paper, we assume that $X$ is sufficiently large, and $\varepsilon>0$ is chosen to be arbitrarily small. From now on, we assume that $f \in \mathcal{E}_{k}.$ For any two functions $k(x):\mathbb{R} \rightarrow \mathbb{R}$ and $l(x):\mathbb{R} \rightarrow \mathbb{R}^{+}$, we employ the notation $k(x)\ll l(x)$ or $k(x)=O(l(x))$ to denote the existence of a positive constant $C$ such that $|k(x)| \leq C l(x)$ for all $x$.
For any set $A,$ we use $1_{A}$ to represent the characteristic function of $A$. For convenience, we denote $B_{f,h}$ by $B_{h}.$
When summing over the index $p$, it implies summation over primes. The term $\log_{2}(x)$ refers to the binary logarithm.
Let 
$$
\begin{aligned}
& Q=X^{1/12-10\varepsilon}, \\
& I:=[0,1], \\
& \mathcal{M}:=\bigcup_{1 \leq q \leq Q} \bigcup_{\substack{1 \leq a \leq q \\
(a, q)=1}}\left(a/q-X^{-\frac{5}{6}-2 \varepsilon}, a/q+X^{-\frac{5}{6}-2 \varepsilon}\right), \\
& m:=\mathcal{M} \backslash I:=I \cap \mathcal{M}^c, \\
& \operatorname{Res}_{s=a} f(s):=\text {The residue of } f \text { at } s=a, \\
& B_h:=\sum_{q=1}^{\infty} \sum_{\substack{1 \leq a \leq q \\
(a, q)=1}}\left(\sum_{q=q_0 q_1 } \frac{\mu\left(q_1\right)}{\phi\left(q_1\right) q_0} w_{f, q_0, q_1}\right)^2 e\left(ah/q\right)
\end{aligned}$$
where $w_{f,q_{0},q_{1}}$ will be defined in \eqref{3wfq}. 
We denote 
$$\begin{aligned}
D[g](s) & :=\sum_{n=1}^{\infty} \frac{g(n)}{n^s}, \\
D[g](s, \chi) & :=\sum_{n=1}^{\infty} \frac{g(n) \chi(n)}{n^s}, \\
D[g](s ; q) & :=\sum_{\substack{n=1 \\
(n, q)=1}}^{\infty} \frac{g(n)}{n^s}, \\
D[g]\left(s, \chi, q_0\right) & :=\sum_{n=1}^{\infty} \frac{g\left(n q_0\right) \chi(n)}{n^s}, \\
D[g]\left(s, q_0 ; q_1\right) & :=\sum_{\substack{n=1 \\
\left(n, q_1\right)=1}}^{\infty} \frac{g\left(n q_0\right)}{n^s}.
\end{aligned}$$

\subsection{The circle method}
Let $$S_{f}(\alpha):= \sum_{X \leq n \leq 2X} f(n)e(n\alpha).$$
Note that $f(n) \leq d_{2}(n)^{k+1}.$
Applying the Hardy-Littlewood circle method and the divisor bound, we see that 
$$\int_{I} |S_{f}(\alpha)|^{2}e(h\alpha)d\alpha = \sum_{X\leq n \leq 2X} f(n)f(n+h)+O\big(h\max_{2X-h \leq n \leq 2X} d_{k}(n)d_{2}(n)d_{k}(n+h)d_{2}(n+h) \big).$$
Considering that $d_{k}(n)d_{2}(n)\ll d_{2}(n)^{k} \ll n^{o(1)},$ we can conclude that the error term in the above equation is bounded by $hX^{o(1)}.$

Now, we present two propositions for major arc and minor arc estimates, which will be proved in Section 3 and Section 4, respectively.

\begin{Prop}\label{Prop1}(Major arc estimate) Let $\varepsilon>0$ be sufficiently small. Take $1\leq H \leq X^{1-\varepsilon}$. Then 
\begin{equation}\int_{\mathcal{M}}|S_{f}(\alpha)|^{2}e(h\alpha)d\alpha-B_{h}X= O_{f,k,\varepsilon}\left(HX^{o(1)}+X^{5/6+2\varepsilon} + X^{1+o(1)}Q^{-1+o(1)}\right). \end{equation}
\end{Prop}
\begin{Prop}\label{Propm}(Minor arc estimate) Let  $\varepsilon>0$  be sufficiently small.  Take $ X^{23/24 +2\varepsilon} \leq H \leq X^{1-\varepsilon}$. Then 
 \begin{equation}\label{Prop2} \int_{m \cap [\theta -\frac{1}{2H}, \theta+\frac{1}{2H}]}
 |S_{f}(\alpha)|^{2} d\alpha \ll_{f,k,\varepsilon}  X^{1-\varepsilon^{2}+o(1)} \nonumber\end{equation} for all $\theta \in [0,1].$
\end{Prop}
\begin{Proof-non} 
The proof of this basically follows from \cite[Proposition 3.1]{MRT1}.   Since $I=\mathcal{M} \cup m,$ we have 
\begin{equation}\sum_{X\leq n \leq 2X} f(n)f(n+h)- \int_{\mathcal{M}}  |S_{f}(\alpha)|^{2}e(h\alpha)d\alpha  
 \end{equation}
\begin{equation}
=O_{\varepsilon}\left(\left|\int_{m}|S_{f}(\alpha)|^{2}e(h\alpha)d\alpha\right| +hX^{\varepsilon^{2}} \right).\nonumber
\end{equation}
    First, we show that for all but $O_{f,k,\varepsilon}\left(X^{-\varepsilon^{2}/3} H\right)$ many integers $h \in [1,H],$  
\begin{equation}
\int_{m} |S_{f}(\alpha)|^{2}e(h\alpha)d\alpha \ll_{f,k,\varepsilon} X^{1-\varepsilon^{2}/4}. \nonumber \end{equation}
Using the Chebyshev inequality, it suffices to show that 
\begin{equation}\label{minorbound}\sum_{1 \leq h \leq H} \left|\int_{m}|S_{f}(\alpha)|^{2}e(h\alpha) d\alpha\right|^{2} \ll_{f,k,\varepsilon} H X^{2-5\varepsilon^{2}/6} .\end{equation}
Let $\Phi(x): \mathbb{R} \rightarrow \mathbb{R}$ be an even non-negative Schwarz function such that 

$$\Phi(x) \geq 1 \thinspace \textrm{\thinspace for \thinspace $x \in [0,1]$},$$
$$\widehat{\Phi}(s):= \int_{\mathbb{R}} \Phi(x)e(-xs)dx =0 \thinspace \thinspace \textrm{except \thinspace for \thinspace $s \in [-\frac{1}{2},\frac{1}{2}]$}.$$ 
By the Poisson summation formula, we have
$$\sum_{h} e(h(\alpha-\beta))\Phi(h/H)= H \sum_{k} \widehat{\Phi}(H(\alpha-\beta-k)).$$
Therefore, after attaching $\Phi(h/H)$ and squaring out, we have 
\begin{equation}\label{eq24}
\begin{split}
\sum_{1 \leq h \leq H} \left|\int_{m}|S_{f}(\alpha)|^{2}e(h\alpha) d\alpha \right|^{2} \ll H\int_{m} |S_{f}(\alpha)|^{2} \int_{m \cap [\alpha-\frac{1}{2H}, \alpha+\frac{1}{2H}]} |S_{f}(\beta)|^{2}d\beta d\alpha. \end{split}\end{equation} 
By applying the divisor bound $f(n) \leq d_{2} (n)^{k+1},$ we have
$$ \sum_{X \leq n \leq 2X} d_{2}(n)^{2k+2} \ll X\prod_{p}\left(1+(2^{2k+2}-1)/p\right).$$
Therefore, by applying Shiu's theorem (see \cite{Shiu1}), we see that
\begin{equation}\label{minorsimple} \int_{m} |S_{f}(\alpha)|^{2}d\alpha \ll \sum_{n=X}^{2X} d_{2}(n)^{2k+2} \ll X (\log X)^{2^{2k+2}-1}.  
    \end{equation}
Combining \eqref{eq24} and \eqref{minorsimple}, the left-hand side of $\eqref{minorbound}$ is bounded by   
\begin{equation}\begin{split}
&\ll_{f} HX(\log X)^{2^{2k+2}-1} \sup_{\alpha\in {m}}\int_{m\cap [\alpha-\frac{1}{2H}, \alpha+\frac{1}{2H}]} |S_{f}(\beta)|^{2}d\beta.
\nonumber
\end{split}
\end{equation}
Assuming Proposition \ref{Propm}, this is bounded by 
$$ \ll_{f,k,\varepsilon} HX^{2-\varepsilon^{2}+o(1)}. $$
Therefore, for all but $O_{f,k,\varepsilon}\left(X^{-\varepsilon^{2}/3} H\right)$ many $h \in [1,H],$
$$\sum_{X\leq n \leq 2X} f(n)f(n+h)- \int_{\mathcal{M}}  |S_{f}(\alpha)|^{2}e(h\alpha)d\alpha  =O_{f,\varepsilon}\left(X^{1-\varepsilon^{2}/4}+HX^{o(1)}\right).$$
By Proposition \ref{Prop1} and the condition $H \ll X^{1-\varepsilon},$ we have 
$$\sum_{X\leq n \leq 2X} f(n)f(n+h)- B_{h}X=O_{f,k,\varepsilon}\big(X^{1-\varepsilon^{2}/4} \big)$$
for all but $O_{f,k,\varepsilon}\left(X^{-\varepsilon^{2}/3} H\right)$ many $h \in [1,H].$ \qed
\end{Proof-non}

\section{\rm{Major arc estimates}}
\noindent In this section, we prove the major arc estimates. 
We apply some analytic properties of L-functions twisted by Dirichlet characters to get the average values of $f(q_{0}n)\chi(n)$ in Lemma \ref{Lemma 3.4}, Lemma \ref{Lemma 3.2}.
Attaching $e(n\beta)$ to $f(n)e(\frac{a}{q}n)$, we get
$$\sum_{X\leq n\leq 2X} f(n)e(\frac{a}{q} n)e(\beta n) \asymp \int_{X}^{2X}\sum_{q=q_{0}q_{1}}\frac{\mu(q_{1})}{\phi(q_{1})q_{0}} w_{f,q_{0},q_{1}}e(\beta x) dx$$ (see Lemma \ref{Lemma 3.8}). Note that the results in this section are very similar to the results in \cite{Kim2022s}. For completeness, we include the proof. From now on we regard $q=q_{0}q_{1}.$ 

By the assumptions, we can use the following lemma.
\begin{Lemma}\label{moving} Assume \eqref{squareup}. Then for any sufficiently large $T$, there exists $T_{0} \in [T,2T]$ 
such that 
$$\sup_{\sigma \in [1/2,1]}\left|L(g,\chi,\sigma+iT_{0})\right| \ll_{\varepsilon} (qT)^{\varepsilon^{2}}$$
\end{Lemma}
\begin{proof}
See \cite[Lemma 2]{Ramsan}.
\end{proof}

\noindent By Remark \ref{Remark 3.3}, we get the following asymptotic formula.

\begin{Lemma}\label{Lemma 3.4} Let  $\varepsilon>0$  be sufficiently small. Then

$$ \sum_{X \leq n\leq 2X\atop(n,q_{1})=1} f(q_{0}n)=w_{f,q_{0},q_{1}}X+O_{f,k,\varepsilon}\left((Xq_{0})^{1/2+\varepsilon} \right)$$
for some $w_{f,q_{0},q_{1}}$.
\end{Lemma}

\begin{proof}
First, we pick $X_{0} \in [X,2X]$ as in Lemma \ref{moving}. Let $$D(s)= \sum_{n=1 \atop (n,q_{1})=1}^{\infty} \frac{f(q_{0}n)}{n^{s}}.$$
By using Perron's formula, we get
$$\sum_{X\leq n\leq 2X   \atop(n,q_{1})=1} f(q_{0}n)  =\frac{1}{2\pi i} \int_{1+\varepsilon^{2}+iX_{0}}^{1+\varepsilon^{2}+iX_{0}} D(s) \frac{(2X)^{s}-X^{s}}{s} ds+O_{f,\varepsilon}(q_{0}^{1/2 + o(1)}X^{\varepsilon}).$$
By shifting the line of integration to the line  $\Re(s)=\frac{1}{2},$ we have
$$\sum_{X\leq n\leq 2X\atop(n,q_{1})=1}f(q_{0}n)=\omega_{f,q_{0},q_{1}}X+ \frac{1}{2\pi i}\int_{\frac{1}{2}-iX_{0}}^{\frac{1}{2}+iX_{0}} D(s)\frac{(2X)^{s}-X^{s}}{s} ds + O_{\varepsilon,f}\left(q_{0}^{1/2+o(1)} X^{\varepsilon}\right)$$
where $$w_{f,q_{0},q_{1}} = \textrm{Res}_{s=1} L(g,s)\zeta(s)  \prod_{p|q} |L_{p}(f,1)| \prod_{p^{l} \parallel q_{0}}(f(p^{l})+f(p^{l+1})p^{-1}+...).$$
By applying H\"{o}lder's inequality, \eqref{q1} and $$\int_{0}^{X_{0}}|\zeta(\frac{1}{2}+it)|^{2}dt \ll_{\varepsilon} X (\log X),$$
we see that 
\begin{equation}\begin{split}\int_{\frac{1}{2}-iX}^{\frac{1}{2}+iX} |\frac{ D(s)}{s}|X^{1/2}ds
&\ll_{f,\varepsilon} q_{0}^{1/2+o(1)} \left(X^{1/2}  \Big|\int_{\frac{1}{2}-iX}^{\frac{1}{2}+iX} \left|\frac{L(g,s)}{s}\right|^{1/2}ds\Big|^{1/2}  \Big|\int_{\frac{1}{2}-iX}^{\frac{1}{2}+iX} \left|\frac{\zeta(s)}{s}\right|^{1/2}ds\Big|^{1/2} + X^{\varepsilon}\right)\\ 
&\ll_{f,k,\varepsilon} ( Xq_{0})^{1/2+\varepsilon}\end{split} \nonumber \end{equation} 
Therefore, we have  
 \begin{equation}\label{smoothing} 
 \sum_{X\leq n\leq 2X \atop(n,q_{1})=1} f(q_{0}n)- w_{f,q_{0},q_{1}}X \ll_{f,k,\varepsilon} ( Xq_{0})^{1/2+\varepsilon}\end{equation} 
 \end{proof}
 
 Note that by using the divisor bound, when $q=q_{0}q_{1} \geq 2,$ it follows that 
 \begin{equation}\label{3wfq}\begin{split} |w_{f,q_{0},q_{1}}| &\leq \textrm{Res}_{s=1}L(g,s)\zeta(s)  (\log_{2}q)^{A_{2}}.
\end{split} \end{equation} for some $A_{2}.$
By applying the same argument as in the proof of Lemma \ref{Lemma 3.4}, we have the following lemma.

\begin{Lemma}\label{Lemma 3.2} Let  $\varepsilon>0$  be sufficiently small. Then for any non-principal character $\chi$ modulo $q_{1},$ 
$$ \sum_{X \leq n\leq 2X \atop(n,q_{1})=1} \chi(n)f(q_{0}n)^{2} \ll_{f,k,\varepsilon}(Xq_{0})^{1/2+\varepsilon}q_{1}^{\varepsilon^{2}}.$$
\end{Lemma}

\noindent To attach $e(\beta n)$ to $f(n)e(an/q),$ we apply the following lemma.  
\begin{Lemma}\label{Lemma 3.8} $($\cite{MRT1}, \rm{Lemma 2.1}$)$
Let $F:[X,2X] \rightarrow \mathbb{C}$ be a smooth function. Then for any function $G:\mathbb{N} \rightarrow \mathbb{C}$ and absolutely integrable function $\mathbf{H}:[X,2X] \rightarrow \mathbb{C}$, 

$$\sum_{X\leq n \leq 2X} F(n)G(n) - \int_{X}^{2X} F(x)\mathbf{H}(x) dx \leq |F(2X)E(2X)|+\int_{X}^{2X} |F'(x)E(x)|dx$$ 
where $$E(x):=\max_{X\leq X' \leq x}\big|\sum_{X\leq n \leq X'}G(n)-\int_{X}^{X'} \mathbf{H}(x)dx\big|.$$
\end{Lemma}
\begin{proof}
This follows easily from summation by parts.
\end{proof}
\begin{Rem}\label{Remark 3.9}
In our case, we take $$F(n)=e(n\beta),\; G(n)=f(n)e(\frac{a}{q}n),\; \mathbf{H}(x)=\sum_{q=q_{0}q_{1}}\frac{\mu(q_{1})}{\phi(q_{1})q_{0}} w_{f,q_{0},q_{1}}.$$ Notice that $\mathbf{H}(x)$ is a constant in $x.$ By using \eqref{additive},  we have  
  \begin{equation}\begin{split} E(x)&=\max_{X \leq X' \leq x}\left|\sum_{q=q_{0}q_{1}}\frac{\mu(q_{1})}{\phi(q_{1})q_{0}} w_{f,q_{0},q_{1}}X'+O_{f,k,\varepsilon}\left(q^{1/2+\varepsilon+o(1)}X^{1/2+\varepsilon}\right)-X'\mathbf{H}(x)\right| \\&= O_{f,k,\varepsilon}\left(q^{1/2+\varepsilon+o(1)}X^{1/2+\varepsilon}\right). \end{split} \end{equation}  for 
$x \in [X,2X].$
\end{Rem}
\noindent Now we are ready to prove Proposition \ref{Prop1}.
\begin{Proof-non4} 
Let $\alpha=\frac{a}{q}+\beta$ for some $ q \leq Q, (a,q)=1.$ 
By Lemma \ref{Lemma 3.4},  Lemma \ref{Lemma 3.2} and Remark \ref{Remark 3.9}, we see that
\begin{equation}\begin{split}S_{f}(\frac{a}{q}+\beta)&=\sum_{X\leq n\leq 2X}f(n)e(\frac{a}{q} n)e(\beta n)\\&=\int_{X}^{2X}\sum_{q=q_{0}q_{1}}\frac{\mu(q_{1})}{\phi(q_{1})q_{0}} w_{f,q_{0},q_{1}}e(\beta x) dx +O_{f,k,\varepsilon}\left((|\beta|+\frac{1}{X})X^{3/2+\varepsilon}q^{1/2+\varepsilon+o(1)}\right). \end{split}\nonumber \end{equation} 
Let $$D_{q}:= \sum_{q=q_{0}q_{1}} \frac{\mu(q_{1})}{\phi(q_{1})q_{0}}w_{f,q_{0},q_{1}}.$$ 
By \eqref{3wfq}, we have
\begin{equation}\label{3dq} \begin{split} 
|D_{q}| &\ll_{f} \sum_{q=q_{0}q_{1}} \frac{1}{\phi(q_{1})q_{0}}|w_{f,q_{0},q_{1}}| \\&\ll_{f} \sum_{q=q_{0}q_{1}} \frac{1}{q} \textrm{Res}_{s=1}L(f,s)(\log_{2} q)^{A_{2}} \\&\ll_{f} \frac{(\log q)^{A_{2}}}{q} \sum_{q=q_{0}q_{1}}1 \\&\ll_{f} \frac{d_{2}(q) (\log q)^{A_{2}}}{q}.
\end{split} \end{equation}
By the Fourier inversion formula, we have
\begin{equation}\label{3FIF} \begin{split}\int_{\mathbb{R}}\big|\int_{X}^{2X} e(\beta x) dx\big|^{2} e(\beta h) d\beta &= \int_{\mathbb{R}} 1_{[X,2X]}(x)1_{[X,2X]}(x+h) dx\\&=X+O(h)\\&=X+O(H).\nonumber\end{split}\end{equation} 
By applying the divisor bound, we have
 $$\int _{|\beta|\leq X^{-\frac{5}{6}-2\varepsilon }}|S_{f}(\frac{a}{q}+\beta)|^{2}e\big(h(\frac{a}{q}+\beta)\big)d\beta=e(h\frac{a}{q})X |D_{q}|^{2}\big(1+O(H/X)\big) $$ 
\begin{equation}\begin{split}
& + O_{f,k,\varepsilon}\bigg(\int_{|\beta|\leq X^{-\frac{5}{6}-2\varepsilon}}\big(|D_{q}|(|\beta|X+1) X^{3/2+\varepsilon}q^{1/2+\varepsilon+o(1)}+  (|\beta|^{2}+\frac{1}{X^{2}})X^{3+2\varepsilon }q^{1+2\varepsilon+o(1)} \big)d\beta \\&+ \int_{|\beta|>X^{-\frac{5}{6}-2\varepsilon}}     \frac{|D_{q}|^{2}}{ |\beta|^{2}}d\beta\bigg) 
    \\&= e(h\frac{a}{q}) |D_{q}|^{2}X +  O_{f,k,\varepsilon}\big(|D_{q}|^{2}H +|D_{q}|\left(X^{5/6-3\varepsilon}q^{1/2+\varepsilon+o(1)}+X^{2/3-\varepsilon}q^{1/2+\varepsilon^{2}+o(1)}\right)
    \\&\thinspace \thinspace \thinspace + X^{1/2-4\varepsilon}q^{1+2\varepsilon+o(1)} +  X^{1/6-\varepsilon} q^{1+2\varepsilon+o(1)} + |D_{q}|^{2}X^{5/6+2\varepsilon}) \nonumber \end{split}\end{equation} 
    Note that $Q = X^{1/12-10\varepsilon}.$
Therefore, we see that
\begin{equation}
\begin{split}
\sum_{1 \leq q \leq Q} \sum_{1 \leq a \leq q \atop(a,q)=1}& \int _{|\beta|\leq X^{-\frac{10}{11}-2\varepsilon} }|S_{f}(\frac{a}{q}+\beta)|^{2}e\big(h(\frac{a}{q}+\beta)\big)d\beta 
\\&=X\sum_{1 \leq q\leq Q} c_{q}(h)|D_{q}|^{2}+O_{f,k,\varepsilon}\left( HX^{o(1)} + X^{5/6-3\varepsilon} +  + X^{5/6+2\varepsilon}\right)\\
&=X\big(B_{h} - \sum_{q> Q} c_{q}(h)|D_{q}|^{2}\big)+O_{f,k,\varepsilon}\left(HX^{o(1)}+X^{5/6+2\varepsilon}\right)
\end{split}\nonumber
\end{equation}
where $c_{q}(h)$ is the Ramanujan sum. It is known that $$|c_{q}(h)|= \left|\mu\left(\frac{q}{(q, h)}\right) \frac{\phi(q)}{\phi\left(\frac{q}{(q, h)}\right)}\right| \leq (q,h).$$  Since $|D_{q}|^{2}\ll_{f} \frac{1}{q^{2-o(1)}},$ we have

\begin{equation} \begin{split} \sum_{q> Q} c_{q}(h)|D_{q}|^{2} &\ll_{f} \sum_{d|h} \sum_{q>Q \atop (h,q)=d} \frac{d}{q^{2-o(1)}}\\ &\ll_{f} \sum_{d|h} \sum_{q>\frac{Q}{d}} \frac{d}{(dq)^{2-o(1)}}\\ &\ll_{f} d_{2}(h)Q^{-1+o(1)}    .\end{split}\end{equation}
 Therefore, we conclude that 
\begin{equation}\label{majorerror}\int_{\mathcal{M}} |S_{f}(\beta)|^{2}e(h\beta)d\beta -B_{h}X=O_{f,k,\varepsilon}\left(HX^{o(1)}+X^{5/6+2\varepsilon} + X^{1+o(1)}Q^{-1+o(1)}\right).\nonumber \end{equation}
 \qed
\end{Proof-non4}

\section{Minor arc estimates }
\noindent In subsection 4.1, we will apply Lemma \ref{Lemma13} to handle $g(n)$ in $S_{f}(\alpha)$ for $n > HQ^{-1/2}X^{-\varepsilon^{2}}$. In subsection 4.2, we will essentially follow the methods applied in \cite{MRT1} (see Lemma \ref{Lemma 4.5}). In subsections 4.3 and 4.4, we will complete the proof of the minor arc estimates in Proposition \ref{Propm}.

\subsection{A modification of $f(n)$}

\noindent To apply the expression $f = g \ast 1$ and combinatorial decomposition, it is necessary to separate $S_{f}(\alpha)$.
\begin{Lemma}\label{Lemma 4.1} Let $1\leq Y \leq 2X.$ Then 
\begin{equation}\label{minormain}\begin{split}\int_{m \cap [\theta -\frac{1}{2H}, \theta+\frac{1}{2H}]}
 &|S_{f}(\alpha)|^{2} d\alpha \ll\int_{m\cap [\theta -\frac{1}{2H}, \theta+\frac{1}{2H}]} \bigg(  \Big|\sum_{Y\leq d \leq 2X} g(d) \sum_{\frac{X}{d} \leq k \leq \frac{2X}{d}}e(dk \alpha)\Big|^{2} \\&
+\Big|\sum_{1 \leq d \leq Y} g(d) \sum_{\frac{X}{d} \leq k \leq \frac{2X}{d}}e(dk \alpha)\Big|^{2}\bigg) d\alpha.
\end{split} \nonumber
\end{equation}
\end{Lemma}
\begin{proof}

By using the fact that $f=g \ast 1,$ we have $$ f(n) = \sum_{d|n} g(d).$$
Therefore, we see that
\begin{equation} 
\begin{split}S_{f}(\alpha) = \sum_{1\leq d \leq 2X} g(d) \sum_{\frac{X}{d} \leq k \leq \frac{2X}{d}} e(dk\alpha). 
\end{split}\nonumber \end{equation} By separating the sum over $d,$ we have
\begin{equation} \begin{split}S_{f}(\alpha)= \sum_{1 \leq d \leq Y} g(d)\sum_{\frac{X}{d} \leq k \leq \frac{2X}{d}}e(dk \alpha) \thinspace +
\sum_{Y\leq d \leq 2X} g(d)\sum_{\frac{X}{d} \leq k \leq \frac{2X}{d}}e(dk \alpha).\end{split} \nonumber\end{equation} The proof is completed by applying the Cauchy-Schwarz inequality.
\end{proof}
\noindent First, we bound the first integral on the right-hand side of \eqref{minormain} by applying Lemma \ref{Lemma13}. Note that by Dirichlet's theorem, if $\alpha \in m \cap [\theta -\frac{1}{2H}, \theta+\frac{1}{2H}],$
then $X^{-5/6-2\varepsilon}<|\alpha-a/q|<1/(qQ)$ for some $a$ and $q$ such that $(a,q)=1$ and $1\leq q \leq Q$.
\begin{Lemma}\label{Lemma 4.2} Let $1\leq Y \leq 2X,$ and let  $X^{-5/6-2\varepsilon}<|\alpha-a/q|<1/qQ$ for some $a$ and $q$ such that $(a,q)=1$ and $1\leq q \leq Q$. Then
\begin{equation}\label{4error1}\begin{split} & \int_{m\cap [\theta -\frac{1}{2H},  \theta+\frac{1}{2H}]}\Big|\sum_{Y \leq d \leq 2X} g(d) \sum_{\frac{X}{d} \leq k \leq \frac{2X}{d}}e(dk \alpha)\Big|^{2} d\alpha 
\\& \ll_{f,k,\varepsilon} q^{1+o(1)}X^{3+2\varepsilon^{2}}Y^{-1}|\alpha-a/q|^{1+2\varepsilon^{2}}H^{-1} + X^{2+20\varepsilon}Y^{-1/3+10\varepsilon}H^{-1}. \end{split}\end{equation}
\end{Lemma}
\begin{proof}

By rearranging the sum, we see that
$$\sum_{Y \leq d \leq 2X} g(d) \sum_{\frac{X}{d} \leq k \leq \frac{2X}{d}}e(dk \alpha) =  \sum_{1\leq k\leq \frac{2X}{Y}} \sum_{\max(\frac{X}{k},Y) \leq d \leq \frac{2X}{k}} g(d) e(dk\alpha).$$
Since $(X/k)^{-5/6-2\varepsilon} \ll k|\alpha-a/q|,$ by the dyadic splitting with applying Lemma \ref{Lemma13}, we see that 
\begin{equation}\label{4pointwise}
\begin{split}
   \sum_{1\leq k\leq \frac{2X}{Y}} &\sum_{{\max(\frac{X}{k},Y)} \leq d \leq \frac{2X}{k}} g(d)e(dk\alpha) \\&\ll_{f,k,\varepsilon} \sum_{1 \leq k \leq \frac{2X}{Y}} \left( q^{1/2+o(1)}\left(2X/k\right)^{1+\varepsilon^{2}} (k |\alpha-a/q|)^{1/2+\varepsilon^{2}} + \left(\frac{2X}{k}\right)^{5/6+5\varepsilon}\right) 
   \\& \ll_{f,k,\varepsilon} q^{1/2+o(1)}X^{3/2 + \varepsilon^{2
   }} |\alpha-a/q|^{1/2+\varepsilon^{2}} Y^{-1/2}+ X^{1+10\varepsilon} Y^{-1/6+5\varepsilon}. 
\end{split} 
\end{equation}

Using \eqref{4pointwise} as a pointwise bound, we see that
\begin{equation}\begin{split} &\int_{m\cap [\theta -\frac{1}{2H}, \theta+\frac{1}{2H}]} \Big|\sum_{Y \leq d \leq 2X} g(d) \sum_{\frac{X}{d} \leq k \leq \frac{2X}{d}}e(dk \alpha)\Big|^{2} d\alpha \\&\ll_{f,k,\varepsilon} q^{1+o(1)}\left(X^{3+2\varepsilon^{2}}Y^{-1} |\alpha-a/q|^{1+2\varepsilon^{2}}\right)  H^{-1} + X^{2+20\varepsilon}Y^{-1/3 +10\varepsilon}  H^{-1}. 
\end{split}\end{equation}
\end{proof}
\noindent Since $|\alpha-a/q|\leq 1/(qQ), $ by taking $Y=HQ^{-1/2}X^{-\varepsilon^{2}},$ the right-hand side of the above inequality is bounded by \begin{equation}\label{firstH}\begin{split} q^{1+o(1)} & X^{3+2\varepsilon^{2}+\varepsilon^{2}} |\alpha-a/q|^{1+2\varepsilon^{2}} H^{-2}  +X^{2+20\varepsilon}H^{-4/3 +10\varepsilon} Q^{1/6+5\varepsilon}  
\\&\ll_{f,k,\varepsilon} X^{3+3\varepsilon^{2}}H^{-2} Q^{-1/2-2\varepsilon^{2}} +X^{2+20\varepsilon}H^{-4/3 +10\varepsilon}Q^{1/6+\varepsilon}.\end{split} \end{equation} 
\begin{Rem}
Since we assume that $X^{23/24+10\varepsilon} \ll H \ll X^{1-\varepsilon},$ the first term is larger than the second term.
\end{Rem}
\subsection{Combinatorial decomposition}
In this subsection, we prove that
\begin{equation} \begin{split}
\int_{m \cap [\theta -\frac{1}{2H}, \theta+\frac{1}{2H}]}  \Big|\sum_{1 \leq d \leq HQ^{-1/2}X^{-\varepsilon^{2}}} g(d)\sum_{\frac{X}{d} \leq k \leq \frac{2X}{d}}e(dk \alpha)\Big|^{2} d\theta \ll_{f,k,\varepsilon}  X^{1-\varepsilon^{2}+o(1)}. 
\end{split}\nonumber
\end{equation}
Let $$b(n):= \sum_{m|n \atop 1\leq m < HQ^{-1/2}X^{-\varepsilon^{2}} }g(m),$$
 $$S_{b}(\alpha) := \sum_{X \leq n \leq 2X} b(n)e(n\alpha).$$ 
We use the following lemma to $S_{b}(\alpha).$ 
\begin{Lemma}\label{Lemma 4.3}
Let $(a,q)=1,$ and let $\gamma, \rho'$ be real numbers such that $|\gamma| \ll \rho' \ll 1.$ Let 
\begin{equation}
I_{\gamma,\rho'}:=  \{ t \in \mathbb{R}: \rho'|\gamma|X \leq |t| \leq \frac{|\gamma|X}{\rho'} \}.
\end{equation}
Then 
$$\int_{\gamma-\frac{1}{H}}^{\gamma+\frac{1}{H}}\big|S_{b}(\frac{a}{q}+\theta)\big|^{2} d \theta$$
 \begin{equation}\label{longintegral}\ll \frac{d_{2}(q)^{4}}{q|\gamma|^{2}H^{2}} \sup_{q=q_{0}q_{1}} \int_{ I_{\gamma,\rho'}} \Big(\sum_{\chi(\mathrm{mod} \thinspace q_{1})}\int_{t-|\gamma|H}^{t+|\gamma|H} \big|\sum_{\frac{X}{q_{0}} \leq n \leq \frac{2X}{q_{0}}} \frac{b(q_{0}n)\chi(n)}{n^{\frac{1}{2}+it'}}\big| dt'\Big)^{2}dt\end{equation} $$ +(\rho'+\frac{1}{|\gamma| H})^{2} \int_{\mathbb{R}}\big(H^{-1} \sum_{x \leq n \leq x+H} |b(n)|\big)^{2} dx.$$

\end{Lemma}
\begin{proof} See \cite[Corollary 5.3]{MRT1}. $b(n)$ can be replaced by any function. \end{proof}
Note that the results in this section are almost identical to those in 
\cite{Kim2022s}. For the sake of completeness, we give the proof. 

\begin{Lemma} $( \thinspace $\cite[Proposition 6,1]{MRT1}, \rm{Combinatorial decomposition}$)$
Let $\varepsilon>0$ be sufficiently small.
Let $X\geq 2$, and set $X^{\frac{23}{24}+10\varepsilon} \ll H \ll X^{1-\varepsilon}.$ 
Set $\rho'=Q^{-\frac{1}{2}}$, and let $1\leq q_{1}\leq Q.$ 
Let $\gamma'$ be a quantity such that $X^{-\frac{5}{6}-2\varepsilon}\leq \gamma' \leq \frac{1}{q_{1}Q}$.
Let $\mathbf{g}:\mathbb{N}\rightarrow \mathbb{C}$ be a function  that takes one of the following forms;
\begin{description} 
\item[\normalfont{Type} $\rm{d_{1}}$ \normalfont{sum} ]
\begin{equation}\label{Typed1} \mathbf{g}= \varrho \ast \varsigma \end{equation}
for some arithmetic function $\varrho: \mathbb{N} \rightarrow \mathbb{C}\thinspace$  where $\varrho$ is supported on $[N,2N]$ and $\varrho(n) \leq d_{2}(n)^{k},$  $\varsigma=1_{(M,2M]}$ for some $N,M$ obeying the bounds
$$ 1 \ll N \ll Q^{1/2}, \; \frac{X}{Q} \ll NM \ll_{\varepsilon} X.$$ 
\item[ \normalfont{Type} $\rm{\RomanNumeralCaps{2}}$ \normalfont{sum} ] 
\begin{equation} \mathbf{g}=\varrho \ast \varsigma \end{equation}
for some arithmetic functions $\varrho, \varsigma: \mathbb{N} \rightarrow \mathbb{C}\thinspace$ where $\max\left(\varrho(n),\varsigma (n)\right) \leq d_{2}(n)^{k}$ and $\varrho$ is supported on $[N,2N]$, $\varsigma$ is supported on $[M,2M]$ for some $N,M$ satisfying the bounds
$$ Q^{1/2} \ll N \ll HX^{-\varepsilon^{2}}Q^{-1/2}, \; \frac{X}{Q} \ll NM \ll_{\varepsilon} X.$$
\end{description}
Then 
\begin{equation}\label{4cd} \int_{I_{\gamma',\rho'}} \Big(\sum_{\chi (\mathrm{mod} \thinspace q_{1})} \int_{t-\gamma' H}^{t+\gamma' H} \big| D[\mathbf{g}](\frac{1}{2}+it',\chi)\big|dt'\Big)^{2}dt \ll_{\varepsilon}q_{1}|\gamma'|^{2}H^{2}X^{1-\varepsilon^{2}+o(1)}.\end{equation}

\end{Lemma}

\begin{proof}
The proof of this basically follows from the arguments in \cite[Proposition 6.1]{MRT1}. Let $J'=[\rho' \gamma' X, \rho'^{-1} \gamma' X].$ Consider the Type $\rm{\RomanNumeralCaps{2}}$ case. Since $g = \varrho \ast \varsigma$, we have
\begin{equation} D[\mathbf{g}](\frac{1}{2}+it',\chi)=D[\varrho](\frac{1}{2}+it',\chi) D[\varsigma](\frac{1}{2}+it',\chi). \end{equation}
By using the Cauchy-Schwarz inequality, we have
\begin{equation}\begin{split}\Big(\sum_{\chi  (\mathrm{mod} \thinspace q_{1})} \int_{t-\gamma' H}^{t+\gamma' H} \big|D[\mathbf{g}](\frac{1}{2}+it',\chi)\big|dt'\Big)^{2} &\ll \Big(\sum_{\chi(\mathrm{mod} \thinspace q_{1})} \int_{t-\gamma' H}^{t+\gamma' H} \big|D[\varrho](\frac{1}{2}+it',\chi)\big|^{2}dt' \Big) \\& \times \Big( \sum_{\chi (\mathrm{mod} \thinspace q_{1})} \int_{t-\gamma' H}^{t+\gamma' H} \big|D[\varsigma](\frac{1}{2}+it',\chi)\big|^{2}dt'\Big). \nonumber \end{split}\end{equation}
Using the mean value theorem (\cite[Lemma 2.10]{MRT1}) and the Fubini theorem, we see that
\begin{equation}\begin{split}
&\sum_{\chi(\mathrm{mod} \thinspace q_{1})} \int_{t-\gamma' H}^{t+\gamma' H} \big|D[\varrho](\frac{1}{2}+it',\chi)\big|^{2}dt' \\&\ll (q_{1}\gamma' H + N) (\log q_{1}\gamma' HN)^{3}\sum_{N \leq n \leq 2N} \frac{(d_{2}(n)^{2k}}{n},\\&
\int_{ I_{\gamma',\rho'}} \sum_{\chi (\mathrm{mod} \thinspace q_{1})} \int_{t-\gamma' H}^{t+\gamma' H} | D[\varsigma](\frac{1}{2}+it',\chi)|^{2}dt'dt \\& \ll (\log q_{1}\gamma' HN)^{3}|\gamma' H|(q_{1}\rho'^{-1}\gamma' X +M)\sum_{M \leq m \leq 2M} \frac{d_{2}(m)^{2k}}{m}.\end{split} \nonumber\end{equation}
Applying Shiu's Theorem (see \cite[Lemma 2.3,(i)]{MRT2}), we have
\begin{equation}\label{logpowerterm} (\log q_{1}\gamma' HN)^{3}\sum_{N \leq n \leq 2N} \frac{d_{2}(n)^{2k}}{n} \ll (\log X)^{2^{2k}+2}. \end{equation}
Therefore, the left-hand side of \eqref{4cd} is bounded by 
\begin{equation}\begin{split} & \ll_{\varepsilon}q_{1}(\rho'^{-1} q_{1} \gamma' + \frac{Q^{\frac{1}{2}}N}{H}+\frac{1}{N}+ \frac{1}{q_{1}\gamma' H})\gamma'^{2}H^{2}X(\log X)^{2^{4k}+4}\\ &\ll_{\varepsilon} q_{1}\left( Q^{-\frac{1}{2}} +\frac{Q^{\frac{1}{2}}N}{H} + \frac{1}{N} + \frac{X^{\frac{5}{6}+2\epsilon}}{H}\right)\gamma'^{2}H^{2}X(\log X)^{2^{4k}+4} \\& \ll q_{1} \gamma'^{2}H^{2}X^{1-\varepsilon^{2}+o(1)}. \end{split}\end{equation} 
Consider the Type $\rm{d_{1}}$ case.
Let $g=\varrho \ast \varsigma \ast 1_{\{1\}}.$ Then $$D[\mathbf{g}](\frac{1}{2}+it', \chi)=D[\varrho](\frac{1}{2}+it', \chi)D[\varsigma](\frac{1}{2}+it', \chi)D[1_{\{1\}}](\frac{1}{2}+it', \chi).$$
By the Cauchy-Schwarz inequality, we have
\begin{equation} \begin{split} 
\Big(\sum_{\chi (\mathrm{mod} \thinspace q_{1})} \int_{t-\gamma' H}^{t+\gamma' H} |&D[\mathbf{g}](\frac{1}{2}+it',\chi)|dt'\Big)^{2} \ll \Big( \sum_{\chi(\mathrm{mod} \thinspace q_{1})} \int_{t-\gamma' H}^{t+\gamma' H} |D[\varrho](\frac{1}{2}+it',\chi)|^{2}dt' \Big) \\&\times \Big(\sum_{\chi (\mathrm{mod} \thinspace q_{1})} \int_{t-\gamma' H}^{t+\gamma' H}  |D[\varsigma](\frac{1}{2}+it',\chi)D[1_{\{1\}}](\frac{1}{2}+it',\chi)|^{2} dt'\Big).  \nonumber
\end{split} \end{equation}
Using the bound 
 \begin{equation} \sum_{\chi(\mathrm{mod} \thinspace q_{1})} \int_{t-\gamma' H}^{t+\gamma' H} |D[\varrho](\frac{1}{2}+it',\chi)|^{2}dt' \ll (q_{1}\gamma' H + N) (\log X)^{2^{2k}+2}, \nonumber\end{equation}
the left-hand side of \eqref{4cd} is bounded by 
$$(q_{1}\gamma' H+N)(\log X)^{2^{2k}+2}\int_{ I_{\gamma',\rho'}} \sum_{\chi  (\mathrm{mod} \thinspace q_{1})} \int_{t-\gamma' H}^{t+\gamma' H} |D[\varsigma](\frac{1}{2}+it',\chi)D[1_{\{1\}}](\frac{1}{2}+it',\chi)|^{2} dt'dt.$$
By the pigeonhole principle, the above term is bounded by 
$$ (q_{1}\gamma' H+N)(\log X)^{2^{2k}+3} \gamma' H\int_{\frac{T}{2}}^{T} \sum_{\chi  (\mathrm{mod} \thinspace q_{1})}|D[\varsigma](\frac{1}{2}+it,\chi)|^{2}|D[1_{\{1\}}](\frac{1}{2}+it,\chi)|^{2} dt $$
for some $T \in J'\:$ where the absolute constant depends on $\varepsilon.$
Using the fourth moment estimate (see \cite[Corollary 2.12]{MRT1}) and $q_{1} \leq Q=X^{1/12-10\varepsilon},$ we see that
\begin{equation} \begin{split}\sum_{\chi  (\rm{mod} \thinspace \it{q_{1}})} \int_{\frac{T}{2}}^{T} |D[\varsigma](\frac{1}{2}+it,\chi)|^{4}dt &\ll_{\varepsilon} q_{1}T(1+\frac{q_{1}^{2}}{T^{2}}+\frac{M^{2}}{T^{4}})e^{(\log q_{1})^{\frac{1}{2}}}(\log q_{1}T)^{404}. \end{split}\end{equation}
By the trivial bound $|D[1_{\{1\}}](\frac{1}{2}+it,\chi)| \leq 1,$ we have
$$\sum_{\chi  (\mathrm{mod} \thinspace q_{1})} \int_{\frac{T}{2}}^{T} |D[1_{\{1\}}](\frac{1}{2}+it,\chi)|^{4}dt \ll_{\varepsilon}  q_{1}T .$$
By applying the Cauchy-Schwarz inequality, we see that
$$ \sum_{\chi  (\mathrm{mod} \thinspace q_{1})} \int_{t-\gamma' H}^{t+\gamma' H}  |D[\varsigma](\frac{1}{2}+it',\chi)|^{2}|D[1_{\{1\}}](\frac{1}{2}+it',\chi)|^{2} dt' \ll_{\varepsilon} q_{1}T(1+\frac{q_{1}}{T}+\frac{M}{T^{2}}) e^{(\log q_{1})^{\frac{1}{2}}/2}(\log X)^{202.5}.$$
From the conditions in Lemma \ref{Lemma 4.1}  and \eqref{Typed1}, we have $$ T \leq \frac{1}{q_{1}Q^{\frac{1}{2}}}X, \thinspace T^{-1} \leq Q^{1/2}X^{5/6+2\varepsilon-1},$$
$$ \gamma' \leq \frac{1}{q_{1}Q}, \thinspace \gamma'^{-1} \leq X^{5/6+2\varepsilon}, $$
$$\; N \leq Q^{1/2}, \thinspace M \leq X.$$
 Therefore, the left-hand side of \eqref{4cd} is bounded by 
\begin{equation}\label{secondH}\begin{split} &\ll_{\varepsilon} q_{1}(q_{1}T+\frac{TN}{\gamma' H})(1+\frac{q_{1}}{T}+\frac{M}{T^{2}})\gamma'^{2} H^{2}  (\log X)^{202.5+2^{2k}+3} e^{(\log X)^{1/2}/2}\\&\ll_{\varepsilon} q_{1}\gamma'^{2}H^{2} (\log X)^{202.5+2^{2k}+2} e^{(\log X)^{1/2}/2} (q_{1}T+ q_{1}^{2}+q_{1}\frac{M}{T}+\frac{TN}{\gamma' H}+ \frac{q_{1}N}{\gamma' H}+\frac{X}{\gamma' HT})
\\&\ll_{\varepsilon} q_{1}\gamma'^{2}H^{2} X^{1-\varepsilon^{2}+o(1)}.\end{split} \nonumber \end{equation} 
\end{proof}

\subsection{Bounding the first term of \eqref{longintegral}}
\begin{Lemma}\label{Lemma 4.5} Let $(a,q)=1,$ and let $\gamma, \rho'$ be real numbers such that \\
$\gamma \in [X^{-\frac{5}{6}-2\varepsilon},$ $\frac{1}{qQ}], \rho'=Q^{-\frac{1}{2}}.$ Let  $X^{\frac{23}{24}+10\varepsilon} \ll H \ll X^{1-\varepsilon},$ and let 
\begin{equation}
I_{\gamma,\rho'}:=  \{ t \in \mathbb{R}: \rho'|\gamma|X \leq |t| \leq \frac{|\gamma|X}{\rho'} \}.
\end{equation}
Then 
 \begin{equation}\begin{split} \frac{d_{2}(q)^{4}}{q|\gamma|^{2}H^{2}}& \sup_{q=q_{0}q_{1}} \int_{ I_{\gamma,\rho'}} \Big(\sum_{\chi(\mathrm{mod} \thinspace q_{1})}\int_{t-|\gamma|H}^{t+|\gamma|H} \big|\sum_{\frac{X}{q_{0}} \leq n \leq \frac{2X}{q_{0}}} \frac{b(q_{0}n)\chi(n)}{n^{\frac{1}{2}+it'}}\big| dt'\Big)^{2}dt \\&\ll_{\varepsilon} X^{1-\varepsilon^{2}+o(1)}.\end{split}\end{equation} 
\end{Lemma}
\begin{proof}

For simplicity, we use $b$ to denote the arithmetic function $n \rightarrow b(n)$ and  $g$ to denote the arithmetic function $n \rightarrow g(n).$ Consider $b1_{(\frac{X}{2},2X]}\thinspace$ instead of $\thinspace b1_{(X,2X]}.$
$b1_{(\frac{X}{2},2X]}$ can be expressed as the following sum
\begin{equation}
\begin{split}
b1_{(\frac{X}{2},2X]}(q_{0}n)= &\sum_{0 \leq l \leq \log_{2}X}  g 1_{[2^{l},2^{l+1}]}\ast 1_{[\frac{X}{2^{l+1}},\frac{X}{2^{l}}]}(q_{0}n)
\\=&\sum_{0\leq l \leq \log_{2}(X^{-\varepsilon^{2}}HQ^{-1/2})}  g 1_{[2^{l},2^{l+1}]} \ast 1_{[\frac{X}{2^{l+1}},\frac{X}{2^{l}}]}(q_{0}n).
\end{split}
\nonumber
\end{equation}
Let $l_{1}$ be the largest integer such that 
$2^{l_{1}} \leq X^{\varepsilon^{2}}$, and let $l_{2}$ be the largest integer such that $2^{l_{2}} \leq X^{-\varepsilon^{2}}HQ^{-1/2}$.  Then \begin{description}
\item[\normalfont{For} $0\leq l \leq l_{1}$]$g 1_{[2^{l},2^{l+1}]}\ast 1_{[\frac{X}{2^{l+1}},\frac{X}{2^{l}}]}$ are Type $\rm{d_{1}}$ sums $\varrho \ast \varsigma$ such that $$\varrho= g 1_{[2^{l},2^{l+1}]},\; \varsigma= 1_{[\frac{X}{2^{l+1}},\frac{X}{2^{l}}]}.$$ 
\item[\normalfont{For} $ l_{1} < l \leq l_{2}$]  $g 1_{[2^{l},2^{l+1}]} \ast 1_{[\frac{X}{2^{l+1}},\frac{X}{2^{l}}]}$ are Type $\rm{\RomanNumeralCaps{2}}$ sums $\varrho \ast \varsigma$ such that $$\varrho= g 1_{[2^{l},2^{l+1}]},\; \varsigma= 1_{[\frac{X}{2^{l+1}},\frac{X}{2^{l}}]}.$$\end{description}
Note that
\begin{equation}\label{4chilambda}\sum_{n=1}^{\infty} \frac{\chi(n)g 1_{[2^{l},2^{l+1}]}\ast  1_{[\frac{X}{2^{l+1}},\frac{X}{2^{l}}]}(q_{0}n)}{n^{\frac{1}{2}+it}}= \sum_{d|q_{0}}\sum_{(n,q_{0})=d} \frac{\chi(n)g1_{[2^{l},2^{l+1}]}\ast  1_{[\frac{X}{2^{l+1}},\frac{X}{2^{l}}]}(q_{0}n)}{n^{\frac{1}{2}+it}}.
\end{equation}
For simplicity, assume that $q_{0}$ is square-free. By using the multiplicativity of $g$, the absolute value of \eqref{4chilambda} is bounded by 
\begin{equation}
\sum_{d|q_{0}} \Big|g1_{[2^{l},2^{l+1}]}\ast  1_{[\frac{X}{2^{l+1}},\frac{X}{2^{l}}]}(dq_{0})\frac{\chi(d)}{d^{\frac{1}{2}+it}}\Big|\Big|\sum_{(k,\frac{q_{0}}{d})=1} \frac{\chi(k)g1_{[2^{l},2^{l+1}]}\ast  1_{[\frac{X}{2^{l+1}},\frac{X}{2^{l}}]}(k)}{k^{\frac{1}{2}+it}}\Big|.   \end{equation}
Using the divisor bound, we have
\begin{equation}\Big|g1_{[2^{l},2^{l+1}]}\ast  1_{[\frac{X}{2^{l+1}},\frac{X}{2^{l}}]}(dq_{0})\frac{\chi(d)}{d^{\frac{1}{2}+it}}\Big| \ll d_{2}(q_{0})^{k+1}. \end{equation}
Therefore, $\sum_{\frac{X}{q_{0}} \leq n \leq \frac{2X}{q_{0}}} \frac{b(q_{0}n)\chi(n)}{n^{\frac{1}{2}+it'}}$ is bounded by a linear combination (with coefficients of size $d_{2}(q_{0})^{k+1}$) of $O\big((\log_{2}X) d_{2}(q_{0})\big)$ absolute values of the following terms
\begin{equation}\label{4conditionk}\Big|\sum_{(k,\frac{q_{0}}{d})=1} \frac{\chi(k)g1_{[2^{l},2^{l+1}]}\ast  1_{[\frac{X}{2^{l+1}},\frac{X}{2^{l}}]}(k)}{k^{\frac{1}{2}+it}}\Big|.\end{equation}
Let $\chi_{\frac{q_{0}}{d},0}$ be the principal character modulo $\frac{q_{0}}{d}.$ By replacing the character $\chi$ with $\chi_{\frac{q_{0}}{d},0}\:\chi,$ we can remove the condition on $k$ under the summation \eqref{4conditionk}. So we use the conductor $\frac{q}{d}$ instead of $q_{1}.$  
Applying Lemma \ref{Lemma 4.3}  on  $\chi_{\frac{q_{0}}{d},0}(k) \chi(k)g1_{[2^{l},2^{l+1}]}\ast  1_{[\frac{X}{2^{l+1}},\frac{X}{2^{l}}]}(k),$ we see that 
\begin{equation}
d_{2}(q_{0})^{2k+2} \int_{ I_{\gamma,\rho'}} \Big(\sum_{\chi (\mathrm{mod}\thinspace q_{1})} \int_{t-\gamma H}^{t+\gamma H} \Big|\sum_{k=1}^{\infty} \frac{\chi_{\frac{q_{0}}{d},0}(k)\chi(k)g1_{[2^{l},2^{l+1}]}\ast  1_{[\frac{X}{2^{l+1}},\frac{X}{2^{l}}]}(k)}{k^{\frac{1}{2}+it}}\Big|dt'\Big)^{2}dt \nonumber\\ \end{equation} $$\ll_{\varepsilon,B}d_{2}(q_{0})^{2k+2}\frac{q}{d}|\gamma|^{2}H^{2}X^{1-\varepsilon^{2}+o(1)}.$$
In the same manner we obtain the same upper bound when $q_{0}$ is not square-free.  
Hence, the first term in Lemma \ref{Lemma 4.3}  is bounded by $X^{1-\varepsilon^{2}+o(1)}.$ \end{proof}
\subsection{Bounding the second term of \eqref{longintegral}}
\begin{Lemma}\label{Lemma 4.6} Let $\gamma \in [X^{-\frac{5}{6}-2\varepsilon}, \frac{1}{qQ}],$  $X^{\frac{23}{24}+10\varepsilon} \ll H \ll X^{1-\varepsilon}, $and let $\rho'=Q^{-\frac{1}{2}}.$ Then
\begin{equation}\label{42}
\Big(\rho' +\frac{1}{|\gamma|H}\Big)^{2}\int_{X}^{2X} \big(\frac{1}{H} \sum_{x \leq n \leq x+H}|b(n)|\big)^{2}dx\ll X^{1-\varepsilon^{2}+o(1)}.
\end{equation}
\end{Lemma}
\begin{proof}
By using the Cauchy-Schwarz inequality, we have
\begin{equation}
\begin{split}
\int_{X}^{2X} \big(\frac{1}{H} \sum_{x \leq n \leq x+H}|b(n)|\big)^{2}dx&\ll\frac{1}{H} \int_{X}^{2X} \sum_{x \leq n \leq x+H}\big|\sum_{m|n}g(m) 1_{[HQ^{-1/2}X^{-\varepsilon^{2}},2X]}(m)\big|^{2} dx
\\&\ll \sum_{X \leq n \leq 2X}\big|\sum_{m|n}g(m)1_{[HQ^{-1/2}X^{-\varepsilon^{2}},2X]}(m)\big|^{2}.
\end{split}
\nonumber
\end{equation}
 Using the inequality $d_{2}(mn) \ll d_{2}(m)d_{2}(n),$ 
\begin{equation}
\begin{split} 
\sum_{X \leq n \leq 2X}\big|\sum_{m|n}g(m)1_{[HQ^{-1/2}X^{-\varepsilon^{2}},2X]}\big|^{2} &\ll \sum_{X \leq n \leq 2X} d_{2}(n) (\sum_{m|n}d_{2}(m)^{2k}1_{[HQ^{-1/2}X^{-\varepsilon^{2}},2X]}(m))
\\& \ll \sum_{HQ^{-1/2}X^{-\varepsilon^{2}} \leq m \leq 2X } d_{2}(m)^{2k} d_{2}(m)\sum_{\frac{X}{m} \leq n \leq \frac{2X}{m}} d_{2}(n)
\\& \ll \sum_{HQ^{-1/2}X^{-\varepsilon^{2}} \leq m \leq 2X} d_{2}(m)^{2k+1}  \frac{X}{m} \log X
\\& \ll X(\log X)^{2} \max_{HQ^{-1/2}X^{-\varepsilon^{2}}\leq M \leq 2X} \frac{\sum_{M \leq m \leq 2M} d_{2}(m)^{2k+1}}{M}.
\end{split}\nonumber\end{equation}
Since \begin{equation}\frac{\sum_{M<m<2M} d_{2}(m)^{2k+1}}{M} \ll (\log M)^{2^{2k+1}-1}, \;\frac{1}{\gamma H} \leq X^{\frac{5}{6}+2\varepsilon}X^{-\frac{23}{24}-10\varepsilon}\ll_{\varepsilon} X^{-\frac{1}{8}-8\varepsilon}, \nonumber \end{equation}
\eqref{42} is bounded by $X^{1-\varepsilon^{2}+o(1)}.$
\end{proof}
\begin{Proof-non5}
Combining Lemma \ref{Lemma 4.2}, Lemma \ref{Lemma 4.5}  and Lemma \ref{Lemma 4.6} , we get 
\begin{equation} \begin{split}\int_{m \cap [\theta -\frac{1}{2H}, \theta+\frac{1}{2H}]}
 |S_{f}(\alpha)|^{2} d\alpha  &\ll_{f,k,\varepsilon} X^{3+3\varepsilon^{2}}H^{-2} Q^{-1-2\varepsilon^{2}}+ X^{1-\varepsilon^{2}+o(1)}. \end{split} \nonumber \end{equation}
 Since $Q=X^{1/12-10\varepsilon}$ and $H \gg X^{23/24+10\varepsilon},$ the left-hand side of the above inequality is bounded by 
 $ X^{1-\varepsilon^{2}+o(1)}.$
 \qed
\end{Proof-non5}

\bibliographystyle{plain}   
\bibliography{over}  

\end{document}